\documentclass{amsart}

\usepackage[margin=1in]{geometry}
\usepackage{amssymb, amsmath}
\usepackage{float}

 \usepackage[abs]{overpic}	
\usepackage{enumerate}

\usepackage{amsthm}

\usepackage{amssymb, amsmath}
\usepackage{float}

\usepackage[colorlinks = true,
            linkcolor = red,
            urlcolor  = red,
            citecolor = red,
            anchorcolor = red]{hyperref}
\usepackage{amsrefs}

\newtheorem{theorem}{Theorem}[section]
\newtheorem{proposition}[theorem]{Proposition}

\newtheorem{remark}[theorem]{Remark}

\begin{document}

   \title{Symmetric quandle colorings and ribbon concordance}
   
   \author{Nicholas Cazet}

   \begin{abstract}

A quandle can always trivially color an orientable surface-link. This note shows that the surface-link $10_1^{-1,-1}$ of Yoshikawa's table cannot be colored by a symmetric dihedral quandle of order 4, and explains how this obstructs a generalized ribbon concordance between another link of two projective planes $8_1^{-1,-1}$ that does admit a coloring by the same symmetric dihedral quandle.

 \end{abstract}

\maketitle

Section \ref{pre} gives preliminary definitions, Section \ref{ch} describes how to color ch-diagrams with symmetric quandles, and Section \ref{obstruct}  uses symmetric quandle coloring to obstruct the existence of a ribbon concordance between the surface-links $8_1^{-1,-1}$ and $10_1^{-1,-1}$ from Yoshikawa's table  \cite{ch} and discusses how the symmetric 3-cocycle invariant can make a similar obstruction.
\begin{figure}

\includegraphics[scale=.65]{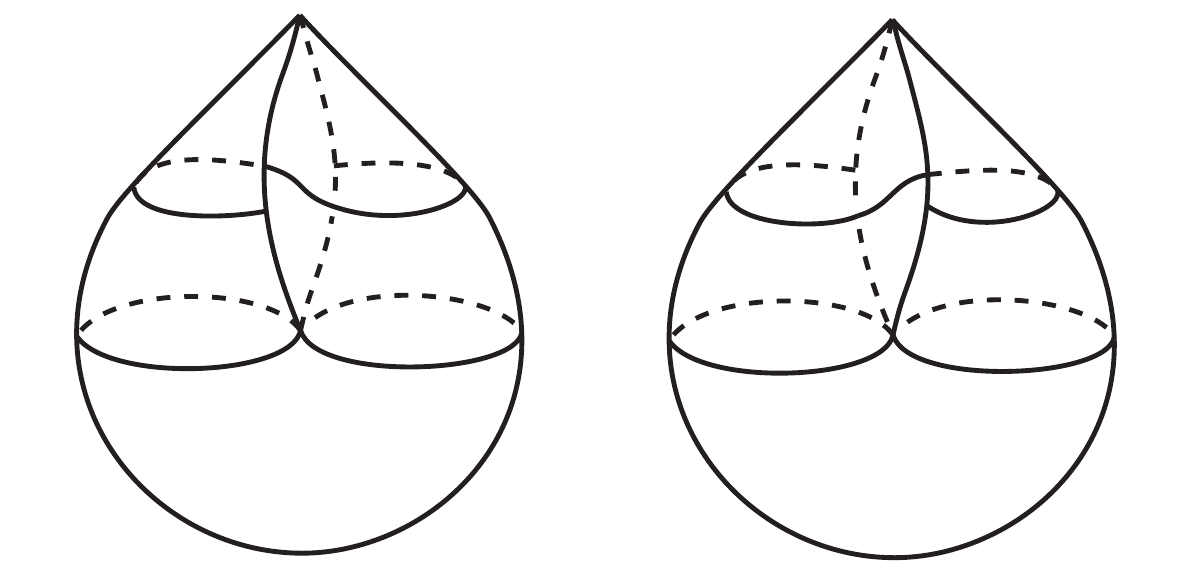}

\caption{Cross-cap broken sheet diagrams of the 2 unknotted projective planes.}

\label{fig:cross}
\end{figure}

\section{preliminaries} 
\label{pre}

A {\it surface-link} is a closed surface smoothly embedded in $\mathbb{R}^4$ up to ambient isotopy. A {\it surface-knot} is a connected surface-link. A {\it 2-knot} is a surface-knot diffeomorphic to $S^2$. A surface-knot is {\it unknotted} if it bounds a handlebody or is the connected sum of some number of unknotted projective planes. Broken sheet diagrams of the two unknotted projective planes, distinguished by their self-intersection numbers or {\it normal Euler numbers}, are shown in Figure \ref{fig:cross}. 

%
%
%
%
%
%
%
%
%
%
%
%

\begin{figure}

\includegraphics[scale=1]{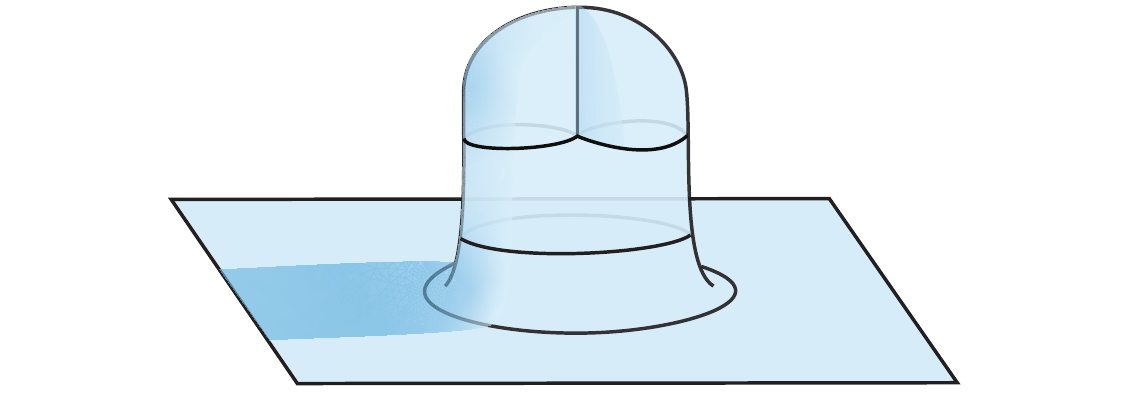}

\caption{Connected sum with an unknotted projective plane $G'\# P$.}

\label{fig:notcolor}
\end{figure}

Let $F_0$ and $F_1$ be surface-links with components of the same genus and orientability. If there is a concordance
$C$ $(\cong F_0\times [0,1]\cong F_1\times [0,1])$ in $S^4\times[0,1]$ between $F_1\subset S^4\times\{1\}$ and $F_0\subset S^4\times\{0\}$ such that the restriction to $C$ of the projection $S^4\times[0,1]\to[0,1]$ is a Morse function with critical points of index 0 and 1 only, then {\it $F_1$ is ribbon concordant to $F_0$ } denoted \[ F_1\geq F_0.\]

\noindent If $F_1 \geq F_0$, then there is a set of $n$ 1-handles on a split union of $F_0$ and $n$ trivial 2-knots such that $F_1$ is obtained by surgeries along these handles. Gordon originally introduced ribbon concordance between classical knots in \cite{gordonribbon}.

  \section{Coloring ch-diagrams}
  \label{ch}

A {\it singular link diagram} is an immersed link diagram in the plane with crossings and traverse double points called {\it vertices}. At each vertex assign a {\it marker}, a local choice of two non-adjacent regions in the complement of the vertex. Such a marked singular link diagram is called a {\it ch-diagram} \cites{CKS,kamada2017surface, kamadakim,ch}. One of the two smoothings of a vertex connects the two regions of the marker, the positive resolution $L^+$, and one separates the marker's regions, the negative resolution $L^-$, see Figure \ref{fig:smoothing}.   If $L^-$ and $L^+$ are unlinks, then the ch-diagram is said to be {\it admissible}. Admissible ch-diagrams represent surface-links and induce broken sheet diagrams and every surface-link defines an admissible ch-diagram.

\begin{figure}

\begin{overpic}[unit=.46mm,scale=.6]{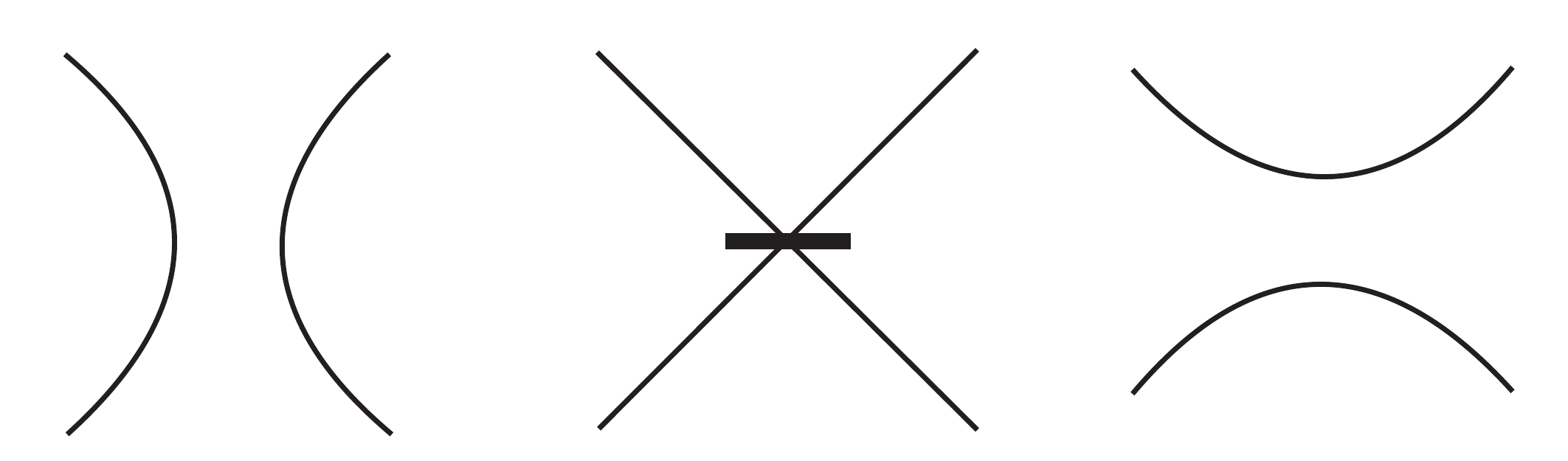}\put(37,-2){$L^-$}\put(233,-2){$L^+$}

\end{overpic}

\caption{Smoothings of a marked vertex.}
\label{fig:smoothing}
\end{figure}

Including the 3 Reidemeister moves of classical link diagrams, there are 8 moves on ch-diagrams called {\it Yoshikawa moves} \cites{ch,kamada2017surface}. Two admissible ch-diagrams represent equivalent surface-links if and only if they are Yoshikawa move equivalent. The ch-index of a ch-diagram is the number of marked vertices plus the number of crossings. The {\it ch-index} of a surface-link $F$, denoted ch($F$), is the minimum of ch-indices ch$(D)$ among all admissible ch-diagrams $D$ representing $F$. A sufrace-link $F$ is said to be {\it weakly prime} if $F$ is not the connected sum of any two surfaces $F_1$ and $F_2$ such that ch$(F_i)$ $<$ ch$(F)$. Yoshikawa classified weakly prime surface-links whose ch-index is 10 or less in \cite{ch}. He generated a table of their representative ch-diagrams up to orientation and mirror. His notation is of the form $I_k^{g_1,\dots,g_c}$ where $I$ is the surface-link's ch-index and $|g_1|,\dots,|g_c|$ are the genera of its components with $g_i<0$ implying the component is non-orientable.

\begin{figure}
\begin{overpic}[unit=.35mm,scale=.8]{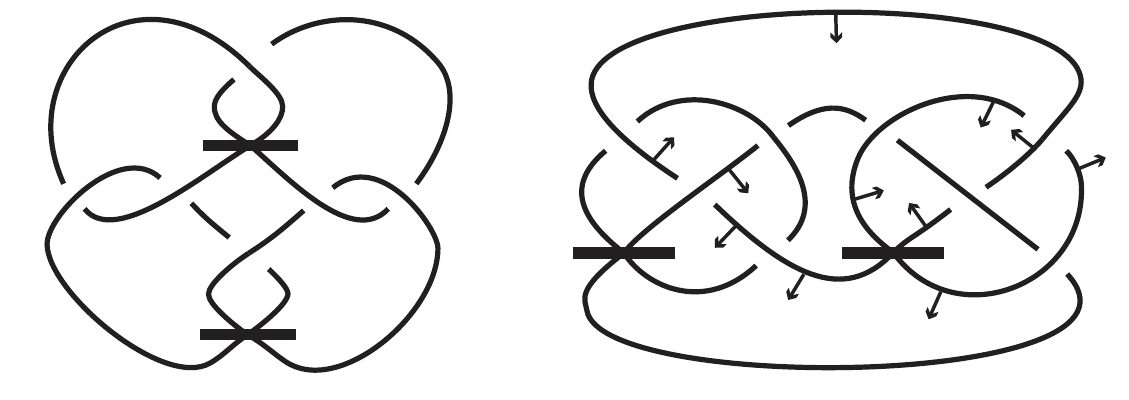}\put(203,81){$b$}\put(200,22){$a$}\put(164,56){$c$}\put(248,58){*}\put(150,43){*}
\end{overpic}

\caption{Ch-diagrams of $8_1^{-1,-1}$ and $10_1^{-1,-1}$.}

\label{fig:notcolors}
\end{figure}
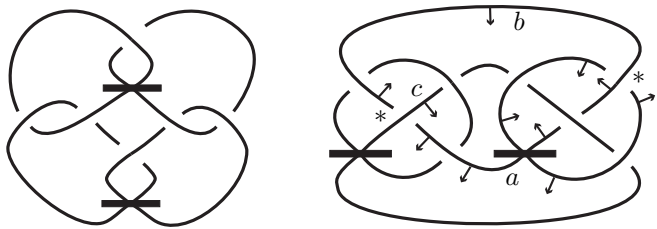

  There is a translation of Reidemeister moves in a motion picture to sheets in a broken sheet diagram. Associate the time parameter of a Reidemeister move with the height of a local broken sheet diagram.  A Reidemeister III move gives a triple point diagram, a Reidemeister I move corresponds to a branch point, and a Reidemeister II move corresponds to a maximum or minimum of a double point curve, see \cites{kamadakim}. Triple points of the induced broken sheet diagram are in correspondence with the Reidemeister III moves in the motion picture. Include sheets containing saddles for each saddle point in the motion picture to generate a broken sheet diagram of the entire surface-link.

 \begin{figure}

\begin{overpic}[unit=.5mm,scale=.6]{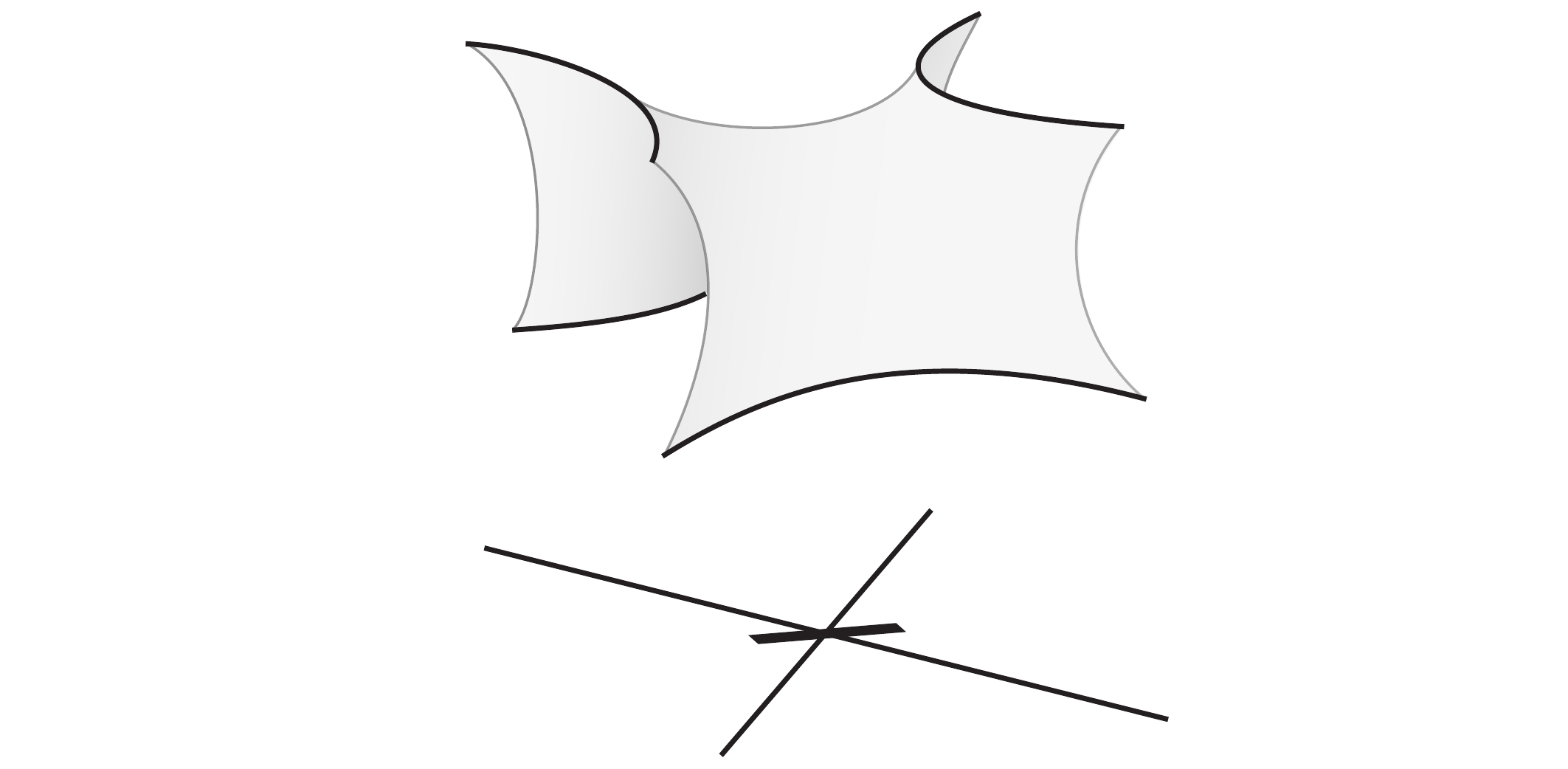}

\end{overpic}

\caption{Induced saddle sheet.}
\label{fig:saddle}
\end{figure}

Consider an admissible ch-diagram. There are finite sequences of Reidemeister moves that take $L^-$ and $L^+$ to crossing-less diagrams $O^-$ and $O^+$. Translate these Reidemeister moves to a broken sheet diagram.  In between the still of $L^-$ and $L^+$ and for each marked vertex include sheets containing the saddles traced by transitioning from $L^-$ to  $L^+$ in the local picture of Figure \ref{fig:smoothing}. These saddles are pictured in Figure \ref{fig:saddle}. Finally,  cap-off $O^-$ and $O^+$ with embedded disks to produce a broken sheet diagram.

An $(X,\rho)$-coloring of a ch-diagram is a $(X,\rho)$-coloring of the link diagram such that the arcs meeting at a vertex are given the same color and their two orientations related by a basic inversion are shown in Figure \ref{fig:orientationvertex}. There is a one-to-one correspondence between colorings of ch-diagram and the colorings of the surface-link it represents. The Reidemeister moves in the trivialization of $L^-$ and $L^+$ uniquely extend the ch-diagram's coloring to the induced broken sheet diagram. Kamada, Kim, and Leem give a detailed account of coloring ch-diagrams and calculating their symmetric quandle cocycle invariant in \cite{kamadakim}.

 \begin{figure}
\centering
\includegraphics[scale=.45]{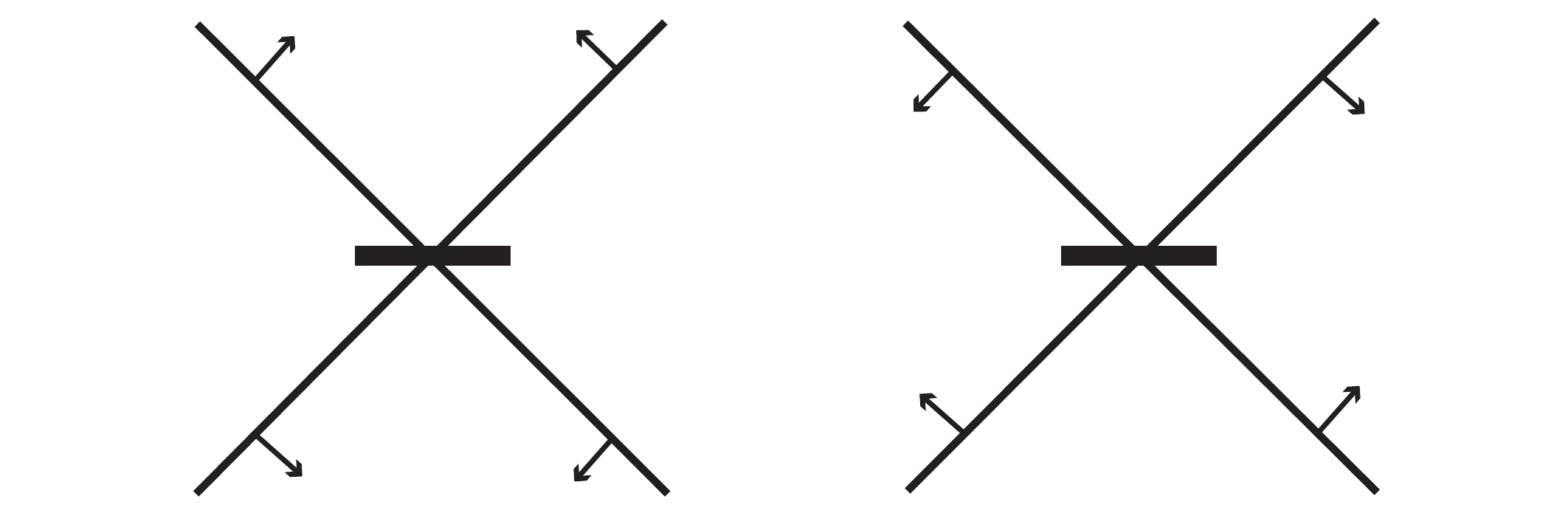}

\caption{Possible orientations at a marked vertex.}
\label{fig:orientationvertex}
  \end{figure}

\section{Symmetric quandles and ribbon concordance}
\label{obstruct}

 Let $D_0$ and $D_1$ be broken sheet diagrams of the surface-knots $F_0$ and $F_1$. Let $\cup_{i=1}^n S_i$ be a disjoint union of broken sheet diagrams each representing an unknotted surface-knot. Suppose that $D_1$ is obtained from $D_0\cup (\cup_{i=1}^n S_i)$ by surgery along sufficiently thin 1-handles $\cup_{j=1}^m h_j$ that intersect $D_0\cup (\cup_{i=1}^n S_i)$ along meridian 2-disks of $h_j$. Then, \[ F_1\succ F_0.\] 

\noindent Figure \ref{fig:general} shows an example. When each $S_i$ is an embedded sphere and each 1-handle $h_j$ attaches $D_0$ to $S_j$ with $n=m$, the relation $\succ$ is ribbon concordance $\geq$. 

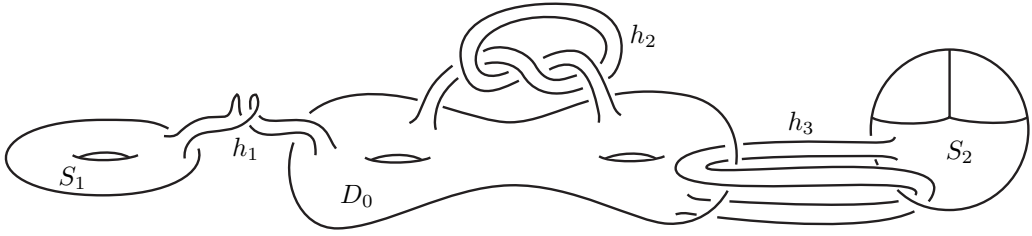
\begin{figure}
\centering
\begin{overpic}[unit=.35mm,scale=.9]{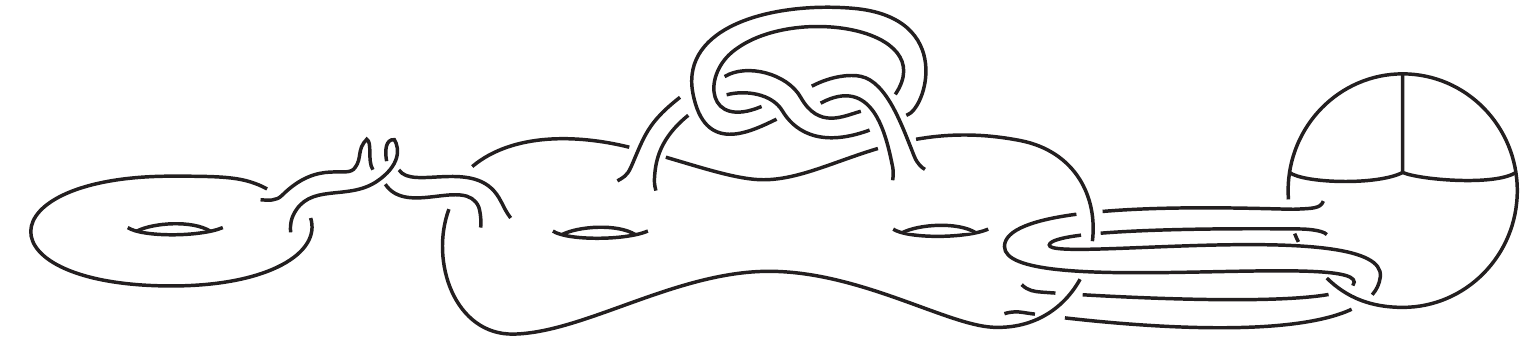}\put(135,12){$D_0$}\put(94,30){$h_1$}\put(28,19){$S_1$}\put(365,29){$S_2$}\put(245,73){$h_2$}\put(305,40){$h_3$}
\end{overpic}
\caption{Handle surgery on $D_0 \cup (\cup_{i=1}^2 S_i)$.}
\label{fig:general}
\end{figure}

\begin{proposition}
For any symmetric quandle $(X,\rho)$ and $(X,\rho)$-set $Y$, if $F_1\succ F_0$ then $F_1$ is $(X,\rho)_Y$-colorable only if $F_0$ is $(X,\rho)_Y$-colorable.
\label{prop:color}
\end{proposition}

\begin{proof}
Suppose that $D_1$ is obtained from $D_0\cup \left(\cup_{i=1}^n S_i\right)$ by surgeries along 1-handles $\cup_{j=1}^m h_j$ that intersect $D_0$ in meridional disks. A $(X,\rho)_Y$-coloring of $D_1$ restricts to a $(X,\rho)_X$-coloring of the punctured diagram $cl(D_0 -\cup_{j=1}^m h_j)$. The coloring of the punctured diagram uniquely extends to a $(X,\rho)_X$-coloring of $D_0$.
\end{proof}

\begin{theorem}

$8_1^{-1,-1}\nsucc10_1^{-1,-1}$

\end{theorem}

\begin{proof}

The symmetric quandle $(R_4,\rho)$ where $\rho(0)=2$ and $\rho(1)=3$ colors $8_1^{-1,-1}$. Kamada and Oshiro used a non-trivial $(R_4,\rho)$-coloring and a symmetric quandle cocycle to calculate the triple point number of $8_1^{-1,-1}$ and its generalizations \cite{sym}.  It will be shown that $(R_4,\rho)$ does not color $10_1^{-1,-1}$, the result then follows from Proposition \ref{prop:color}.

Note the arbitrary colors $a,b,c\in R_4$ and orientations given to $10_1^{-1,-1}$ in Figure \ref{fig:notcolors}. Two relations can be read from the starred crossings, $a^b=\rho(a)$ and $b^c=\rho(a)$. Using the definition of the quandle operation for $R_4$, $a^b=2b-a=\rho(a)$ and $b^c=2c-b=\rho(a)$. The first relation is equivalent to $2b=a+\rho(a)$ which is either 0 or 2. If $a=0$ or $a=2$, then $a+\rho(a)=2$ and $b$ must be odd. If $a=1$ or $a=3$, then $a+\rho(a)=0$ and $b$ must be even. Therefore, the colors $a,b\in R_4$ cannot both be even or odd.

The second relation is equivalent to $2c=b+\rho(a)$. The involution $\rho$ does not change the parity of the color. Therefore, $b+\rho(a)$ must be odd since $a$ and $b$ differ in parity. This contradicts the evenness of $b+\rho(a)$ from $2c=b+\rho(a)$. Thus, no such $a,b,c\in R_4$ exist and $(R_4,\rho)$ cannot color $10_1^{-1,-1}$.

\end{proof}

The symmetric quandle 3-cocycle invariant can also be used to obstruct ribbon concordance. For every symmetric quandle $(X,\rho)$ and $(X,\rho)$-set $Y$ there is an associated chain and cochain complex. {\it Symmetric quandle 3-cocycles} are cocycles of this cochain complex and represent cohomology classes of $H^3_{Q,\rho}(X;A)_Y$ for any abelian group $A$,  see  \cites{carter2009symmetric, kamada2017surface, sym}.

 Let $C$ be an $(X,\rho)_Y$-coloring of a broken sheet diagram $D$. For a triple point $t$ of $D$, choose 1 of the 8 3-dimensional complementary regions around $t$ and call the region  {\it specified}. There are 12 semi-sheets around a triple point. Let $S_B$, $S_M$, and $S_T$ be the 3 of them that face the specified region, where $S_B$, $S_M$, and $S_T$ are semi-sheets of the bottom, middle, and top sheet respectively. Let $n_B$, $n_M$, and $n_T$ be the normal orientations of $S_B$, $S_M$, and $S_T$. Through basic inversions, it is assumed that each normal orientation points away from the specified region. Let $x_1,x_2,$ and $x_3$ be the elements of $X$ assigned to the semi-sheets $S_B$, $S_M$, and $S_T$ whose normal orientations $n_B$, $n_M$, and $n_T$ point away from the specified region. Additionally, let $y$ be the color of the specified region. The {\it color} of the triple point $t$ is  $C_t=(y,x_1,x_2,x_3).$ 
 
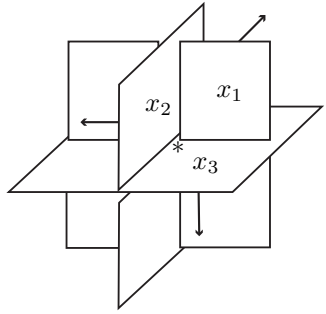
\begin{figure}
\centering
\begin{overpic}[unit=.35mm,scale=.35]{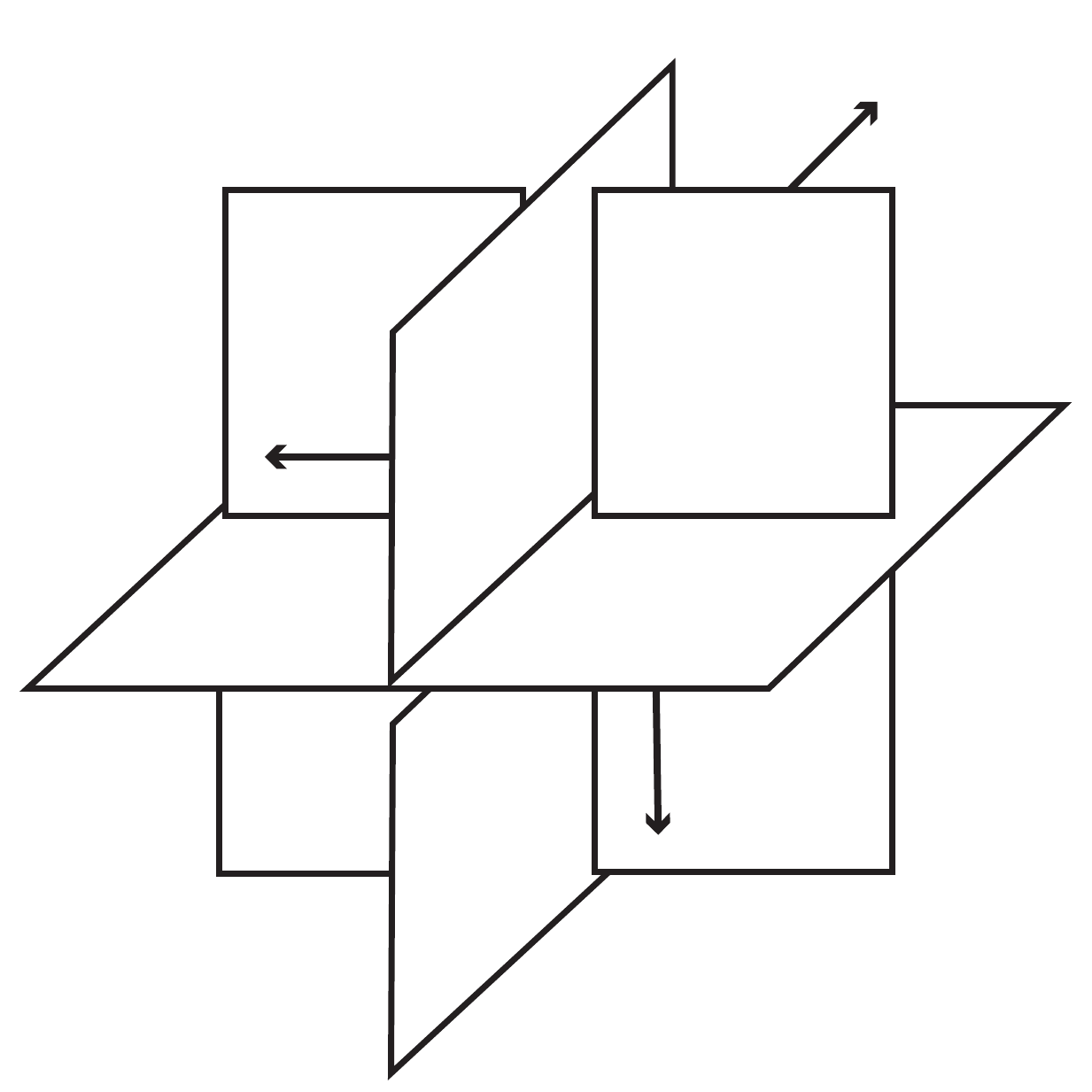}\put(82,82.5){$x_1$}  \put(55,77){$x_2$}  
\put(73,55){$x_3$}\put(65,58){*}
\end{overpic}
\caption{A positive, colored triple point.}
\label{fig:spec}
\end{figure}

 Let $\phi:\mathbb{Z}(Y\times X^3)\to A$ be a symmetric quandle 3-cocycle of $(X,\rho)_Y$ with coefficients in $A$. For a $(X,\rho)_Y$-coloring $C$, the ${\it \phi-weight}$ of the triple point $t$ is defined by $\varepsilon \phi(y,x_1,x_2,x_3)$ such that $\varepsilon$ is +1 (or -1) if the triple of the normal orientations $(n_T,n_M,n_B)$ is (or is not) coherent with the orientation of $S^3\subset\mathbb{R}^4$ at the triple point and is denoted $W_\phi(t;C)$. The triple point of Figure \ref{fig:spec} is positive. The ${\it \phi-weight}$ of a diagram $D$ with respect to a symmetric quandle coloring $C$ is \[ W_\phi(D;C)=\sum_\tau W_\phi(t;C) \in A,\] where $t$ runs over all triple points of $D$. The value $W_\phi(D;C)$ is an invariant of an $(X,\rho)_Y$-colored surface-link $(F,C)$ \cites{sym,oshiro}. Denote $W_\phi(D;C)$ by $W_\phi(F;C)$.

 The {\it symmetric quandle cocycle invariant} of a broken sheet diagram $D$ representing a surface-link $F$ is the multi-set \[ \Phi_\phi(D)=\{W_\phi(D;C) : C\in \text{Col}_{(X,\rho)_X}(D)\}.\] The multi-set $\Phi_\phi(D)$ is an invariant of the $(X,\rho)_Y$-colored surface-link $F$ \cite{sym}.  Since the invariant is independent of broken sheet representative, denote $\Phi_\phi(D)$ by $\Phi_\phi(F)$.

\begin{theorem}

If $F_1\succ F_0$, then $\Phi_\phi(F_1)\stackrel{m}\subset \Phi_\phi(F_0)$ for any symmetric quandle 3-cocycle $\phi$.

\end{theorem}

\begin{proof}

For any element $a\in \Phi_\phi(F_1)$, there is an $(X,\rho)_Y$-coloring $C_1\in Col_{(X,\rho)}(D_1)$ with $a=W_\phi(C_1)=\sum_t W_\phi(t;C_1)$ on $D_1$. Since the intersection of $D_0$ and each 1-handle consists of small 2-disk, the $(X,\rho)_Y$-coloring $C_1$ restricted to the punctured diagram $D_0-(\cup_{j=1}^n h_j)$ determines the $(X,\rho)$-coloring of $D_0$ uniquely, $C_0\in \text{Col}_{(X,\rho)}(D_0)$. Since the set of triple points of $D_1$ is coincident with that of $D_0$, and since $W_\phi(t;C_0)=W_\phi(t;C_1)$ for any triple point $t$, $a=\sum_t W_\phi(t;C_0)=W_\phi(D_0;C_0)\in\Phi_\phi(D_0).$

\end{proof}

The Kinoshita conjecture posits that all knotted projective planes have a trivial projective summand.

\begin{remark}

If the Kinoshita conjecture is true, then a symmetric quandle whose good involution has no fixed points cannot color a projective plane.

\end{remark}

\begin{proof}
The cross-cap diagrams of the unkotted projective planes can only be monochromatically colored with a fixed point of the good involution.
\end{proof}

  \section*{Acknowledgements}

  I would like to thank Maggie Miller for helpful discussion and guidance.

        \bibliographystyle{amsplain}
\bibliography{proposal.bib}

\end{document}